\documentclass[11pt]{amsart}
\usepackage{amsfonts}
\usepackage{amsmath}
\usepackage{amssymb}
\usepackage{graphicx}
\usepackage[margin=1.5in]{geometry}

\newtheorem{theorem}{Theorem}
\newtheorem{lemma}{Lemma}
\newtheorem{corollary}{Corollary}
\newtheorem{proposition}{Proposition}
\theoremstyle{definition}

\newtheorem{claim}{Claim}

\theoremstyle{remark}

\numberwithin{equation}{section}

\begin{document}
\title[Relation between Lebesgue summability and other summation methods]{On the relation between Lebesgue summability and some other summation methods}
\author[J. Vindas]{Jasson Vindas}
\address{Department of Mathematics, Ghent University, Krijgslaan 281 Gebouw S22, B 9000 Gent, Belgium}
\email{jvindas@cage.Ugent.be}

\subjclass[2010]{Primary 40E05, 42A24. Secondary 40G05, 40G10}
\keywords{Fourier series; Lebesgue summability; Ces\`{a}ro summability; Abel summability; Riemann summability; $(\gamma,\kappa)$ summability; Tauberian theorems}

\date{}

\begin{abstract}
It is shown that if 
$$
\sum_{n=1}^{N}n\left|c_{n}\right|=O(N)\:,
$$ then Lebesgue summability, $(\mathrm{C},\beta)$ summability ($\beta>0$), Abel summability, Riemann summability, and $(\gamma,\kappa)$ summability ($\kappa\geq 1$) of the series $\sum_{n=0}^{\infty}c_{n}$ are all equivalent to one another.
\end{abstract}

\maketitle

\section{Introduction}
\label{ls}
In this article we establish the equivalence between various methods of summability under a certain hypothesis (condition (\ref{lseq5}) below). Our results extend a recent theorem of M\'{o}ricz \cite[Thm. 1]{moricz2012}.

We are particularly interested in Lebesgue summability, a summation method that is suggested by the theory of trigonometric series \cite{zygmund}. Consider the formal trigonometric series
\begin{equation}
\label{lseq1}
\frac{a_{0}}{2}+\sum_{n=1}^{\infty}\left(a_{n}\cos nx +b_{n}\sin nx\right)\: .
\end{equation}
Formal integration of (\ref{lseq1}) leads to
\begin{equation}
\label{lseq2}
L(x)=\frac{a_{0}x}{2}+\sum_{n=1}^{\infty}\left(\frac{a_{n}}{n}\sin nx-\frac{b_{n}}{n}\cos nx\right)\: .
\end{equation}
One then says that the series (\ref{lseq1}) is Lebesgue summable at $x=x_{0}$ to  $s(x_{0})$ if (\ref{lseq2}) is convergent in a neighborhood of $x_{0}$ and 
\begin{equation}
\label{lseq3}
s(x_0)=\lim_{h\to0}\frac{\Delta L(x_{0};h)}{2h}\: ,
\end{equation}
where
\begin{align*}
\frac{\Delta L(x_{0};h)}{2h}&=\frac{ L(x_{0}+h)-L(x_{0}-h)}{2h}
\\
&=\frac{a_{0}}{2} +\sum_{n=1}^{\infty} \left(a_{n}\cos nx_{0} +b_{n}\sin nx_{0}\right)\frac{\sin nh}{nh}\:.
\end{align*}
In such a case one writes
$$
\frac{a_0}{2}+\sum_{n=1}^{\infty}\left(a_{n}\cos nx_{0}+b_{n}\sin nx_{0}\right)=s(x_{0}) \ \ \ (\mathrm{L})\:.
$$
Observe that (\ref{lseq3}) tells that the symmetric derivative of the function $L$ exists and equals $s(x_{0})$ at the point $x=x_0$. 

The Lebesgue method of summation is somehow complicated, since it is not regular. In fact, if the series (\ref{lseq1}) converges at $x=x_{0}$, then it is not necessarily Lebesgue summable at $x=x_0$. 

Zygmund investigated conditions under which Lebesgue summability is equivalent to convergence \cite[pp. 321-322]{zygmund}. Among other things, he proved the following result. Set $\rho_{n}=\sqrt{|a_{n}|^{2}+|b_{n}|^{2}}$.

\begin{theorem}[Zygmund] \label{lsth1} If
\begin{equation}
\label{lseq4}
\rho_{n}=O(1/n)\:,
\end{equation}
then, the series $a_0/2+\sum_{n=1}^{\infty} \left(a_{n}\cos nx +b_{n}\sin nx\right)$ is convergent at $x=x_{0}$ to $s(x_0)$ if and only if it is Lebesgue summable at $x=x_{0}$ to $s(x_0)$.
\end{theorem}

M\'{o}ricz has recently studied the role of a certain weaker condition than (\ref{lseq4}) in Lebesgue summability. He has complemented Theorem \ref{lsth1} by showing \cite[Thm. 1]{moricz2012}:

\begin{theorem}[M\'{o}ricz]\label{lsth2} Suppose that
\begin{equation}
\label{lseq5}
\sum_{n=1}^{N}n\rho_{n}=O(N)\: .
\end{equation}
If the series  $a_0/2+\sum_{n=1}^{\infty} \left(a_{n}\cos nx +b_{n}\sin nx \right)$ converges at $x=x_{0}$ to $s(x_0)$, then it is also Lebesgue summable at $x=x_{0}$ to $s(x_0)$.
\end{theorem}

We have found here that, under condition (\ref{lseq5}), not only may Lebesgue summability be concluded from much weaker assumptions than convergence, but also it becomes equivalent to a number of familiar summability methods. In particular, we shall prove the following theorem, which considerably improves Theorem \ref{lsth2} and may be interpreted as a Tauberian theorem relating various summability procedures. In the next statement $(\mathfrak{R},1)$ and $(\mathfrak{R},2)$ denote the Riemann summability methods \cite[Sect. 4.17]{hardy}, while $(\mathrm{C},\beta)$ stands for Ces\`{a}ro summability.

\begin{theorem}
\label{lsth3}
Suppose that (\ref{lseq5}) is satisfied. Then, the following statements are equivalent. The trigonometric series $a_{0}/2+\sum_{n=1}^{\infty}\left(a_{n}\cos nx +b_{n}\sin nx\right)$ is:
\begin{itemize}
\item [(i)] Lebesgue summable at $x=x_{0}$ to $s(x_0)$.
\item [(ii)]Abel summable at $x=x_{0}$ to $s(x_0)$.
\item [(iii)] $(\mathrm{C},\beta)$ summable, for $\beta>0$, at $x=x_{0}$ to $s(x_0)$.
\item [(iv)] $(\mathfrak{R},1)$ at $x=x_{0}$ to $s(x_0)$.
\item [(v)] $(\mathfrak{R},2)$ at $x=x_{0}$ to $s(x_{0})$.
\end{itemize}
\end{theorem}

As shown in Section \ref{The main result}, an additional summability method, which naturally generalizes the Riemann methods $(\mathfrak{R},1)$ and $(\mathfrak{R},2)$, can also be added to the list from Theorem \ref{lsth3} (the so-called $(\gamma,\kappa)$ summability, $\kappa\geq 1$, introduced and studied by Guha in  \cite{guha}). Furthermore, it should also be noticed that Theorem \ref{lsth3} includes Theorem \ref{lsth1} as a particular instance, as immediately follows from, say, Hardy's elementary Tauberian theorem for $(\mathrm{C},1)$ summability \cite{hardy,korevaarbook}.

Theorem \ref{lsth3} actually admits a generalization to the class of summability methods discussed in Section \ref{summability methods}. The main result of this paper is Theorem \ref{lsth4}, stated in Section \ref{The main result}. In Section \ref{The main result} we will also obtain extensions of Theorem \ref{lsth1}. 

It should be mentioned that Theorem \ref{lsth3} intersects with the work of Jakimovski. In fact, the equivalence between (ii) and (iii), under the even weaker condition
\begin{equation}
\label{lseq1.6}
\sum_{n=1}^{N}n\left(a_{n}\cos nx_{0}+b_{n}\sin nx_{0}\right)=O(N)\: ,
\end{equation}
 was already established by him  (cf. \cite[Thm. 6.2]{jakimovski}). On the other hand, the equivalence between the rest of summability methods from Theorem \ref{lsth3} (and that of Guha) appears to be new in the literature. It is also worth pointing out that Jakimovski extensively investigated in \cite{jakimovski} Tauberian theorems for Abel-type methods and Borel summability in which the conclusion is Ces\`{a}ro summability. His Tauberian conditions are in terms of growth comparisons between different higher order Ces\`{a}ro means, or more generally suitable Hausdorff transforms, of the series. Jakimovski's Tauberian conditions may be regarded as average generalizations of (\ref{lseq1.6}). Our approach in this article differs from that developed in \cite{jakimovski}. In the proof of our main result, Theorem \ref{lsth4}, we shall exploit some results by Estrada and the author connecting summability of Fourier series and integrals with local behavior of Schwartz distributions \cite{estrada-vindas2012Taub,vindas-estrada2007,vindas-estrada2010Inv} (cf. \cite{walter1}).

Finally, we mention that Pati \cite{pati2005} and \c{C}anak et al \cite{canak-e-t} have recently made use of
Tauberian conditions involving Ces\`{a}ro average versions of (\ref{lseq1.6}) in the study of Tauberian theorems for the so-called $(\mathrm{A})(\mathrm{C},\alpha)$ summability.

\section{Summability methods}
\label{summability methods}
We collect here the summability methods that will be studied in Section \ref{The main result}. Let $\left\{\lambda_{n}\right\}_{n=0}^{\infty}$ be an increasing sequence of non-negative real numbers tending to infinity. 

We begin with Riesz summability \cite{hardy}. Let $\beta\geq 0$. We say that the series $\sum_{n=0}^{\infty}c_{n}$ is $(\mathrm{R},\left\{\lambda_{n}\right\},\beta)$ summable to $\ell$ if 
$$
\ell=\lim_{x\to\infty} \sum_{\lambda_{n}\leq x} c_n \left(1-\frac{\lambda_{n}}{x}\right)^{\beta}\: .
$$ 
In such a case, we write
\begin{equation}
\label{lseq2.1}
\sum_{n=0}^{\infty}c_{n}=\ell \ \ \ (\mathrm{R},\left\{\lambda_{n}\right\},\beta)\: .
\end{equation}
In the special case $\lambda_{n}=n$, the summability (\ref{lseq2.1}) is equivalent to Ces\`{a}ro $(\mathrm{C},\beta)$ summability, as follows from the well known equivalence theorem of Marcel Riesz \cite{hardy,ingham1968}.

The extended Abel summation method is defined as follows \cite{hardy}. We say that the series $\sum_{n=0}^{\infty}c_{n}$ is $(\mathrm{A},\left\{\lambda_{n}\right\})$ summable to $\ell$ if $\sum_{n=0}^{\infty} c_n e^{-\lambda_n y}$ converges for $y>0$ and
$$
\ell=\lim_{y\to0^{+}} \sum_{n=0}^{\infty} c_n e^{-\lambda_n y}\: ;
$$ 
we then write
\begin{equation}
\label{lseq2.2}
\sum_{n=0}^{\infty}c_{n}=\ell \ \ \ (\mathrm{A},\left\{\lambda_{n}\right\})\: .
\end{equation}
When $\lambda_n=n$, one recovers the usual Abel summability method $(\mathrm{A})$ in (\ref{lseq2.2}).

We shall also consider a generalization of Guha's method from \cite{guha}. We need to introduce the so-called Young functions \cite{hobsonII}. They are given by the Ces\`{a}ro (integral) means of $\cos x$. Let $\kappa\geq 0$. We set $\gamma_{0}(x)=\cos x$ and, for $\kappa>0$,
\begin{equation}
\label{lseq2.3}
\gamma_{\kappa}(x)= \frac{\kappa}{x}\int_{0}^{x} \left(1-\frac{t}{x}\right)^{\kappa-1}  \cos t\: \mathrm{d}t\: .
\end{equation} 
It is said that $\sum_{n=0}^{\infty}c_n$ is $(\gamma, \left\{\lambda_n\right\},\kappa)$ summable to $\ell$ if the following two conditions hold:
\begin{equation*}
\sum_{n=0}^{\infty}c_{n} \gamma_{\kappa}(\lambda_{n}h ) \ \ \ \mbox{converges for small } h>0\:,
\end{equation*}
and
\begin{equation*}
\ell=\lim_{h\to 0^{+}}\sum_{n=0}^{\infty}c_{n} \gamma_{\kappa}(\lambda_{n}h )\: .
\end{equation*}

We employ the notation 
\begin{equation}
\label{lseq2.4}
\sum_{n=0}^{\infty}c_{n}=\ell \ \ \ (\gamma,\left\{\lambda_{n}\right\},\kappa)\: 
\end{equation}
to denote $(\gamma, \left\{\lambda_n\right\},\kappa)$ summability. If $\lambda_n=n$, we write $(\gamma,\kappa)$ instead of $(\gamma, \left\{n\right\},\kappa)$, in accordance with Guha's notation \cite{guha}. As explained in \cite{guha}, the $(\gamma,\kappa)$ method is intimately connected with certain aspects of the theory of summability of trigonometric series. For instance, if $\kappa=1,2$, one obtains in (\ref{lseq2.3}) the functions 
$$
\gamma_{1}(x)=\frac{\sin x}{x} \ \ \ \mbox{and} \ \ \ \gamma_2(x)=\left(\frac{\sin (x/2)}{x/2}\right)^{2};
$$
so that $(\gamma,1)=(\mathfrak{R},1)$ and $(\gamma,2)=(\mathfrak{R},2)$. We recall that $(\mathfrak{R},1)$ and $(\mathfrak{R},2)$ stand for the Riemann summability methods \cite{hardy}. 

Lebesgue summability is of course closely related to the $(\gamma,1)$ method, but observe that the convergence of (\ref{lseq2}) is not part of the requirements for $(\gamma,1)$ summability. In analogy to the Lebesgue summability method, we say that $\sum_{n=0}^{\infty}c_{n}$ is $(\mathrm{L},\left\{\lambda_n\right\})$ summable to $\ell$ and write $\sum_{n=0}^{\infty}c_{n}=\ell$ $(\mathrm{L},\left\{\lambda_n\right\})$ if (\ref{lseq2.4}) holds with $\kappa=1$ and additionally 
\begin{equation}
\label{lseq2.5}
\sum_{0<\lambda_n } c_n \frac{e^{i\lambda_n h}}{\lambda_n h} \ \ \ \mbox{converges for small } |h|>0\: . 
\end{equation}
We point out that our convention for this generalization of Lebesgue summability is different from that proposed by Sz\'{a}sz in \cite[p. 394]{szasz1945}. (In fact, Sz\'{a}sz' notion coincides with what we call here $(\gamma, \left\{\lambda_n\right\},1)$ summability.)

We have

\begin{proposition}
\label{lsp1} Suppose that
\begin{equation}
\label{lseq2.6}
\sum_{\lambda_n\leq x}\lambda_n |c_n|=O(x)\:.
\end{equation}
 Then, the series $\sum_{n=0}^{\infty}c_n$ is $(\mathrm{L},\left\{\lambda_n\right\})$ summable if and only if it is $(\gamma,\left\{\lambda_n\right\},1)$ summable.
\end{proposition}

Thus, under condition (\ref{lseq5}), the trigonometric series (\ref{lseq1}) is Lebesgue summable at $x=x_{0}$ to $s(x_0)$ if and only if it is $(\mathfrak{R},1)$ $(=(\gamma,1))$ summable at $x=x_{0}$ to $s(x_0)$. Proposition \ref{lsp1} follows at once from the ensuing simple lemma, which guarantees the absolute and uniform convergence of (\ref{lseq2.5}) when (\ref{lseq2.6}) is assumed.

\begin{lemma}
\label{lsl1} The condition (\ref{lseq2.6}) is equivalent to
\begin{equation}
\label{lseq2.7}
\sum_{ x\leq  \lambda_n} \frac{|c_{n}|}{\lambda_n}=O\left(\frac{1}{x}\right)\:. 
\end{equation}
\end{lemma}   
\begin{proof} Write $S(x)= \sum_{\lambda_n\leq x}|c_{n}|$ for $x>0$ and $S(0)=0$. The conditions (\ref{lseq2.6}) and (\ref{lseq2.7}) take the form
 \begin{equation}
\label{lseq2.8}
T_{1}(x):=\int_{0}^{x}t\:\mathrm{d}S(t)=O\left(x\right) 
\end{equation}
and
\begin{equation}
\label{lseq2.9}
T_{2}(x):=\int_{x}^{\infty}t^{-1}\mathrm{d}S(t)=O\left(\frac{1}{x}\right)\:, 
\end{equation} 
respectively. Assume (\ref{lseq2.8}). Notice that
$$
\int_{x}^{y}t^{-1}\mathrm{d}S(t)= \int_{x}^{y} t^{-2}\mathrm{d} T_{1}(t)= \frac{T_1(y)}{y^{2}}-\frac{T_1(x)}{x^{2}}+2\int_{x}^{y} \frac{T_{1}(t)}{t^{3}}\:\mathrm{d}t\:.
$$
Taking $y\to\infty$, we obtain that
$$
\int_{x}^{\infty}t^{-1}\mathrm{d}S(t)=-\frac{T_1(x)}{x^{2}}+2\int_{x}^{\infty} \frac{T_{1}(t)}{t^{3}}\:\mathrm{d}t\: =O\left(\frac{1}{x}\right)\:.
$$
Suppose now that (\ref{lseq2.9}) holds. Since 
$$
T_{3}(x):=\int_{(x,\infty)} t^{-1}\mathrm{d} S(t)\leq T_{2}(x)=O(1/x)\:,
$$ 
we have
$$
\int_{0}^{x}t\:\mathrm{d}S(t)= - \int_{0}^{x}t^{2}\mathrm{d}T_{3}(t)=-x^{2}T_{3}(x)+2\int_{0}^{x}tT_{3}(t)\:\mathrm{d}t= O(x)\:,  
$$
as required.
\end{proof}

\section{Main result}
\label{The main result}
We are now in the position to state our main result:
\begin{theorem}
\label{lsth4}
If the condition $(\ref{lseq2.6})$ holds, then the following six statements are equivalent. The series $\sum_{n=0}^{\infty}c_{n}$ is:
\begin{itemize}
\item [(a)]  $(\mathrm{L}, \left\{\lambda_{n}\right\})$ summable to $\ell$.
\item [(b)]  $(\gamma, \left\{\lambda_{n}\right\},\kappa)$ summable to $\ell$ for some $\kappa\geq 1$.
\item [(c)] $(\gamma, \left\{\lambda_{n}\right\},\kappa)$ summable to $\ell$ for all $\kappa\geq 1$.
\item [(d)] $(\mathrm{R}, \left\{\lambda_{n}\right\},\beta)$ summable to $\ell$ for some $\beta> 0$.
\item [(e)] $(\mathrm{R}, \left\{\lambda_{n}\right\},\beta)$ summable to $\ell$ for all $\beta>0$.
\item [(f)] $(\mathrm{A}, \left\{\lambda_{n}\right\})$ summable to $\ell$.
\end{itemize}

\end{theorem}

Before giving a proof of Theorem \ref{lsth4}, we would like to discuss two corollaries of it. It is well known that any of the following three assumptions is a Tauberian condition for $(\mathrm{A}, \left\{\lambda_{n}\right\})$ summability, and hence for Riesz $(\mathrm{R}, \left\{\lambda_{n}\right\},\beta)$ summability,
\begin{equation}
\label{lseq3.1}
c_n=O\left(\frac{\lambda_{n}-\lambda_{n-1}}{\lambda_{n}}\right)\:,
\end{equation}

\begin{equation}
\label{lseq3.2}
\sum^{\infty}_{n=1}\left(\frac{\lambda_n}{\lambda_{n}-\lambda_{n-1}}\right)^{p-1}\left|c_n\right|^{p}<\infty \ \ \ (1<p<\infty)
\:,
\end{equation}

\begin{equation}
\label{lseq3.3}
\sum^{N}_{n=1}\lambda_n^{p}\left(\lambda_{n}-\lambda_{n-1}\right)^{1-p}\left|c_n\right|^{p}=O(\lambda_N) \ \ \ (1<p<\infty)\:.
\end{equation}
Indeed, that convergence follows from $(\mathrm{A}, \left\{\lambda_{n}\right\})$ summability under (\ref{lseq3.1}) was first shown by Ananda Rau in \cite{ananda-rau} (see also \cite{hardy,vindas-estrada2008Taub}).  The Tauberian theorem related to (\ref{lseq3.2}) belongs to Hardy and Littlewood, while the one with the Tauberian condition (\ref{lseq3.3}) to Sz\'{a}sz (see \cite[Sect. 5]{estrada-vindas2012Taub} for quick proofs of these two Tauberian theorems).

We can deduce from Theorem \ref{lsth4} the following Tauberian theorem for $(\gamma, \left\{\lambda_{n}\right\},\kappa)$ summability.

\begin{corollary}
\label{lsc1} Let $\kappa\geq1$. Suppose that
$$
\sum_{n=0}^{\infty}c_{n}=\ell \ \ \ (\gamma, \left\{\lambda_{n}\right\},\kappa)\:.
$$
Then, any of the Tauberian conditions (\ref{lseq3.1}), (\ref{lseq3.2}), or (\ref{lseq3.3})  implies that $\sum_{n=0}^{\infty}c_{n}$ is convergent to $\ell$. 
\end{corollary}
\begin{proof}
Clearly, (\ref{lseq3.1}) yields (\ref{lseq2.6}). Furthermore, any of the two conditions (\ref{lseq3.2}) or (\ref{lseq3.3}) also implies (\ref{lseq2.6}), as a straightforward application of the H\"{o}lder inequality shows. By Theorem \ref{lsth4}, we obtain that the series is $(\mathrm{A}, \left\{\lambda_{n}\right\})$ summable to $\ell$. Consequently, the desired convergence conclusion follows from the corresponding Tauberian theorem for $(\mathrm{A}, \left\{\lambda_{n}\right\})$ summability. 
\end{proof}

Combining Corollary \ref{lsc1} and Theorem \ref{lsth4}, we obtain the ensuing extension of Zygmund's result (Theorem \ref{lsth1}). 

\begin{corollary}
\label{lsc2} Assume any of the conditions (\ref{lseq3.1})--(\ref{lseq3.3}). Then, $\sum_{n=0}^{\infty}c_{n}$ is $(\mathrm{L}, \left\{\lambda_{n}\right\})$ summable to $\ell$ if and only if it is convergent to $\ell$.
\end{corollary}

We now set the ground for the proof of Theorem \ref{lsth4}. The space $\mathcal{S}'(\mathbb{R})$ denotes the well known Schwartz space of tempered distributions \cite{estrada-kanwal,p-s-v}. We will make use of the notion of distributional point values, introduced by \L ojasiewicz in \cite{lojasiewicz}. A distribution $f\in\mathcal{S}'(\mathbb{R})$ is said to have a distributional point value $\ell$ of order $k\in\mathbb{N}$ at the point $x=x_0$ if there is a locally bounded function $F$ such that $F^{(k)}=f$ near $x=x_0$ and 
\begin{equation*}
\lim_{x\rightarrow x_{0}}\frac{k!F(x)}{\left(  x-x_{0}\right)  ^{k}}=\ell\: .
\end{equation*} 
In such a case one writes $f(x_0)=\ell$, distributionally of order $k$. 

We are more interested in the closely related notion of (distributionally) symmetric point values and its connection with the Fourier inversion formula for tempered distributions \cite[Sect. 6]{vindas-estrada2010Inv} (cf. \cite[Chap. 5]{p-s-v}). We say that $f$ has a symmetric point value $\ell$ of order $k$ at $x=x_0$ and write $f_{\textnormal{sym}}(x_0)=\ell$, distributionally of order $k$, if the distribution

\begin{equation}
\label{lseq3.4}
\chi_{x_0}(h):= \frac{f(x_0+h)+f(x_0-h)}{2}
\end{equation}
satisfies $\chi_{x_0}(0)=\ell$, distributionally of order $k$. One can show \cite[Thm. 5.18]{p-s-v} that $f_{\textnormal{sym}}(x_0)=\ell$, distributionally, if and only if the pointwise Fourier inversion formula
\begin{equation}
\label{lseq3.5}
\frac{1}{2\pi}\:\mathrm{p.v.}\left\langle \hat{f}(u),e^{ix_0u}\right\rangle=\ell \ \ \  (\mathrm{C}, \beta)\: 
\end{equation}
holds for some $\beta\geq 0$. The left hand side of (\ref{lseq3.5}) denotes a principal value distributional evaluation in the Ces\`{a}ro sense, explained, e.g., in \cite[Sect. 5.2.8]{p-s-v}. Under additional assumptions on the growth order of $f$ at $\pm\infty$, it is possible to establish a more precise link between the order of summability $\beta$ and the order of the symmetric point value \cite{vindas-estrada2010Inv}. We refer to \cite{estrada-vindas2010,estrada-vindas2012Taub,estrada-vindas2013Taub,p-s-v,vindas-estrada2007,vindas-estrada2010Inv} for studies about the interplay between local behavior of distributions and summability of series and integrals.

We now proceed to show our main result.

\begin{proof}[Proof of Theorem \ref{lsth4}] The equivalence between (d), (e), and (f) has been established by Estrada and the author in \cite[Cor. 4.16]{estrada-vindas2012Taub} under the still weaker assumption
$$
\sum_{\lambda_{n}\leq x}\lambda_{n}c_{n}=O(x)\:.
$$
(The case $\lambda_n=n$ of this result is due to Jakimovski \cite[Thm. 6.2]{jakimovski}.)
Taking Proposition \ref{lsp1} into account, it therefore suffices to show the implications (b)$\Rightarrow$(d) and  (e)$\Rightarrow$(c). 
We first need to show the following claim:

\begin{claim}\label{lsclaim} Let $\kappa\geq1$. Under the assumption (\ref{lseq2.6}),
\begin{equation}
\label{lseq3.6}
 \sum_{n=0}^{\infty}c_n=\ell \ \ \ (\gamma,\left\{\lambda_n\right\},\kappa) \ \  \Longrightarrow  \ \  \sum_{n=0}^{\infty}c_n=\ell \ \ \ (\gamma,\left\{\lambda_n\right\},\tau) \ \ \ \mbox{for }\tau\geq \kappa\: .
\end{equation}
\end{claim}
\begin{proof}[Proof of Claim \ref{lsclaim}] Set $G_{\tau}(h)=\sum_{n=0}^{\infty}c_n\gamma_{\tau}(\lambda_n h)$. Lemma \ref{lsl1} ensures that all these series are absolutely convergent for $h>0$ if $\tau\geq 1$. Let $\tau>\kappa$. Since 
$$\Gamma(\kappa+1)x^{\tau}\gamma_{\tau}(x)=\frac{\Gamma(\tau+1)}{\Gamma(\tau-\kappa)}\int_{0}^{x}(x-t)^{\tau-\kappa-1}t^{\kappa}\gamma_{\kappa}(t)\:\mathrm{d}t\:, $$
we have
\begin{align*}
G_{\tau}(h)&
=\frac{\Gamma(\tau+1)}{\Gamma(\kappa+1)\Gamma(\tau-\kappa)h^{\tau}}\sum_{n=0}^{\infty}c_n\int_{0}^{h}(h-t)^{\tau-\kappa-1}t^{\kappa}\gamma_{\kappa}(\lambda_nt)\:\mathrm{d}t
\\
&
=\frac{\Gamma(\tau+1)}{\Gamma(\kappa+1)\Gamma(\tau-\kappa)h^{\tau}}\int_{0}^{h}(h-t)^{\tau-\kappa-1}t^{\kappa}G_{\kappa}(t) \:\mathrm{d}t
\\
&
=
\frac{\Gamma(\tau+1)}{\Gamma(\kappa+1)\Gamma(\tau-\kappa)}\int_{0}^{1}(1-t)^{\tau-\kappa-1}t^{\kappa}G_{\kappa}(ht) \:\mathrm{d}t
\\
&=\ell+o(1)\:, \ \ \ h\to0^{+}\:, 
\end{align*}
where we have used Lemma \ref{lsl1} and the bound $\gamma_{\kappa}(x)=O(1/x)$ to exchange integration and summation in the second equality.
\end{proof}

We aboard the proof of $(b)\Rightarrow(d)$. Define the tempered distribution
$$
f(x)=\sum_{n=0}^{\infty}c_n e^{i\lambda_nx}\:.
$$ 
By (\ref{lseq3.6}), we can assume that the series is $(\gamma,{\lambda_n},k)$ summable to $\ell$ for an integer $k\geq1$, namely,
\begin{equation}
\label{lseq3.7}
F(h):=\frac{h^{k}}{k!}\sum_{n=0}^{\infty}c_n \gamma_{k}(\lambda_nh)= \ell\:\frac{h^{k}}{k!}+o(|h|^{k})\:, \ \ \ h\to0\: .
\end{equation}
It is clear that $F^{(k)}=\chi_{0}$, where $\chi_{0}$ is the distribution given by (\ref{lseq3.4}). Thus, (\ref{lseq3.7}) leads to the conclusion $f_{\textnormal{sym}}(0)=\ell$, distributionally of order $k$. Therefore, applying \cite[Thm 5.18]{p-s-v}, we obtain
\begin{align*}
\ell&= \frac{1}{2\pi}\:\mathrm{p.v.}\left\langle \hat{f}(u),1\right\rangle \ \ \ \ \  (\mathrm{C}, \beta)
\\
&
=
\lim_{x\to\infty}\sum_{\lambda_n\leq x}c_n \ \ \ \ \  (\mathrm{C}, \beta)
\\
&
= \sum_{n=0}^{\infty}c_n \ \ \ \ \ (\mathrm{R},\left\{\lambda_n\right\},\beta)\:,
\end{align*}
for some $\beta>0$. (It actually follows from the stronger result \cite[Thm. 6.7]{vindas-estrada2010Inv} that this holds for every $\beta>k$.) Hence, the summability (d) has been established.

We now prove (e)$\Rightarrow$(c). We will actually show that if $\sum_{n=0}^{\infty}c_n$ is $(\mathrm{R},\left\{\lambda_n\right\},1)$ summable to $\ell$, then the series is $(\gamma,\left\{\lambda_n\right\},1)$ summable. By (\ref{lseq3.6}), (c) will automatically follow. We may assume that $\ell=0$. Set $S(x)=\sum_{\lambda_n\leq x}c_n$ for $x>0$ and $S(0)=0$. Our assumption is then
\begin{equation*}
S_1(x)= \int_{0}^{x}S(t)\:\mathrm{d}t= o(x)\:, \ \ \ x\to\infty\: . 
\end{equation*}
Employing (\ref{lseq2.6}), we obtain
\begin{equation*}
|S(x)|= \left|\frac{S_{1}(x)}{x}+ \frac{1}{x}\int_{0}^{x}t\:\mathrm{d}S(t)\right|= O\left(1\right)\:.
\end{equation*}
Let $\mu$  and $y$ be two positive numbers to be chosen later. We keep $h<\mu/y$. Write 
$$
\sum_{n=0}^{\infty}c_n \gamma_{1}(\lambda_{n}h)= \left(\sum_{\lambda_{n}\leq \mu/h}+\sum_{\mu/h<\lambda_{n}}\right)c_n \gamma_{1}(\lambda_{n}h)=:I_{1}(h,\mu)+ I_{2}(h,\mu)\:.
$$
By using Lemma \ref{lsl1}, we can estimate $I_{2}(h,\mu)$ as 
\begin{equation*}
|I_{2}(h,\mu)|\leq \frac{1}{h}\sum_{\mu/h<\lambda_{n}} \frac{|c_n|}{\lambda_{n}}< \frac{C_1}{\mu}\:,
\end{equation*} 
where $C_1$ does not depend on $h$. Integrating by parts twice, we get
\begin{align*}
I_{1}(h,\mu)&= \left(\gamma_{1}(\mu) S(\mu/h)-h\gamma_{1}'(\mu)S_1({\mu/h})\right)+h^{2}\left(\int_{0}^{y}+\int_{y}^{\mu/h}\right)S_{1}(t)\gamma_1''(ht)\: \mathrm{d}t 
\\
& 
=: I_{1,1}(h,\mu)+h^{2}I_{1,2}(h,\mu,y)+h^{2}I_{1,3}(h,\mu,y)\: .
\end{align*} 
We can find constants $C_2,C_3,C_4,C_5>0$, independent of $h$, $\mu$, and $y$, such that
\begin{equation*}
|I_{1,1}(h,\mu)|< \frac{C_{2}}{\mu}+ C_3 \frac{\left|S_{1}(\mu/h)\right|}{\mu/h}\:,
\end{equation*}

\begin{equation*} |I_{1,2}(h,\mu,y)|< C_4 h^{2}y^{2}\:,
\end{equation*}
and 
\begin{equation*}
|I_{1,3}(h,\mu,y)|< C_5 h^{2}\int_{y}^{\mu/h} |S_{1}(t)|\: \mathrm{d}t\:\leq C_5 h\mu \max_{t\in[y,\mu/h]}|S_{1}(t)|\: .
\end{equation*}
Given $\varepsilon>0$, we fix $\mu$ larger than $4(C_1+C_2)/\varepsilon$. Next, we can choose $y$ such that $|S_{1}(x)|\leq \varepsilon x/(4\max\{C_3,\mu^{2}C_5\})$ for all $x\geq y$. Finally, if we choose $h_{0}<\min\{\mu/y, \sqrt{\varepsilon/(4C_4y^{2})} \}$, we obtain 
$$
\left|\sum_{n=0}^{\infty}c_n \gamma_{1}(\lambda_{n}h)\right|<\varepsilon \ \ \ \mbox{for } 0<h<h_0\: .
$$
This completes the proof of Theorem \ref{lsth4}.
\end{proof}

\end{document}